\documentclass[a4paper,11pt]{amsart}

\usepackage{amsmath,amsthm,amssymb,mathtools,tikz,enumitem,comment,hyperref}
\usepackage[margin=3cm]{geometry}

\title{Higher-rank GBS groups: non-positive curvature and biautomaticity}
\author{Sam Shepherd}
\address[S.\ Shepherd]{Institut für Mathematische Logik und Grundlagenforschung, Fachbereich Mathematik und Informatik, Universität Münster, Einsteinstraße 62, 48149 Münster, Germany}
\email{sam.shepherd@uni-muenster.de}
\author{Motiejus Valiunas}
\address[M.\ Valiunas]{Instytut Matematyczny, Uniwersytet Wroc{\l}awski, plac Grunwaldzki 2, 50-384 Wroc{\l}aw, Poland}
\email{motiejus.valiunas@math.uni.wroc.pl}

\newcommand{\GLM}[1][n]{rank~$#1$ GBS}

\newcommand{\EE}{{\mathbb{E}}}
\newcommand{\FF}{{\mathbb{F}}}
\newcommand{\GG}{{\mathcal{G}}}
\newcommand{\LL}{{\mathcal{L}}}
\newcommand{\NN}{{\mathbb{N}}}

\newcommand{\QQ}{{\mathbb{Q}}}
\newcommand{\RR}{{\mathbb{R}}}
\newcommand{\TT}{{\mathcal{T}}}
\newcommand{\ZZ}{{\mathbb{Z}}}

\DeclareMathOperator{\AGL}{AGL}
\DeclareMathOperator{\Aut}{Aut}
\DeclareMathOperator{\Cay}{Cay}
\DeclareMathOperator{\CAT}{CAT}
\DeclareMathOperator{\Comm}{Comm}
\DeclareMathOperator{\GL}{GL}
\DeclareMathOperator{\id}{id}
\DeclareMathOperator{\OO}{O}

\theoremstyle{plain}
\newtheorem{thm}{Theorem}[section]

\newtheorem{lem}[thm]{Lemma}

\theoremstyle{definition}

\theoremstyle{remark}
\newtheorem{rem}[thm]{Remark}

\begin{document}

\begin{abstract}
    We characterise when a rank $n$ generalised Baumslag--Solitar group is CAT(0) and when it is biautomatic.
\end{abstract}
\maketitle

\section{Introduction}

Since the development of the theory of biautomatic groups in the late 1980s, the relationship between biautomaticity and various notions of non-positive curvature has been studied.  In particular, it had been open for a long time whether every $\CAT(0)$ group is biautomatic.  Recently, Leary and Minasyan \cite{LearyMinasyan} gave the first examples of CAT(0) groups that are not biautomatic.  Their examples can be characterised as commensurating HNN extensions of free abelian groups, and as part of their work, Leary and Minasyan gave precise criteria for such groups to be $\CAT(0)$ and to be biautomatic.

In this note, we generalise the results of Leary and Minasyan to all \emph{rank $n$ generalised Baumslag--Solitar groups} (or \emph{\GLM{} groups} for short)---fundamental groups of (finite) graphs of groups $\GG$ with all vertex and edge groups isomorphic to $\ZZ^n$ for a fixed $n \in \NN$.  (We call such a graph of groups $\GG$ a \emph{\GLM{} graph of groups}.)  
Rank 1 GBS groups are well studied, for instance from the point of view of splittings, quasi-isometries, automorphisms and commensurability \cite{Whyte01,Forester03,Forester06,Levitt07,Levitt15,CasalRuizKazachkovZakharov21}.
The general setting of \GLM{} groups has also been studied by several authors in recent years, from the perspective of actions on hyperbolic spaces in \cite{Button25} and the perspective of separability properties in \cite{LopezShepherd}.

The criteria for our theorem are based on the image of the \emph{modular homomorphism} $\Delta_\GG$ as a subgroup of the abstract commensurator $\Comm(H) \cong \GL_n(\QQ)$, where $H$ is a vertex group of $\GG$.  See Section~\ref{sec:prelim} for the precise terminology and notation.

\begin{thm} \label{thm:main}
    Let $\GG$ be a \GLM{} graph of groups, and let $\Delta_\GG \colon \pi_1(\GG) \to \GL_n(\QQ)$ be the modular homomorphism. Then,
    \begin{enumerate}[label=(\roman*)]
        \item \label{it:main-biaut} the following are equivalent:
        \begin{enumerate}[label=(\alph*)]
            \item \label{iit:main-biaut} $\pi_1(\GG)$ is biautomatic,
            \item \label{iit:main-vbiaut} $\pi_1(\GG)$ is virtually biautomatic,
            \item \label{iit:main-sbiaut} $\pi_1(\GG)$ is a subgroup of a biautomatic group,
            \item \label{iit:main-finite} the image of $\Delta_\GG$ is finite;
        \end{enumerate}
        \item \label{it:main-CAT0} $\pi_1(\GG)$ is $\CAT(0)$ $\Leftrightarrow$ the image of $\Delta_\GG$ is conjugate in $\GL_n(\RR)$ into $\OO(n)$.
    \end{enumerate}
\end{thm}

After writing the first version of this paper, we learnt that part \ref{it:main-CAT0} of Theorem \ref{thm:main} had been independently proved by Button and was due to appear in his upcoming publication \cite{Button25}.

Note that, as a consequence of Theorem~\ref{thm:main}, any biautomatic \GLM{} group is $\CAT(0)$, since any finite subgroup of $\GL_n(\RR)$ is conjugate to a subgroup of $\OO(n)$.  Indeed, if $G < \GL_n(\RR)$ is finite, then we can define a $G$-invariant inner product on $\RR^n$ by setting $\langle x,y \rangle \coloneqq \sum_{A \in G} \langle Ax,Ay \rangle_0$, where $\langle{-},{-}\rangle_0$ is the standard inner product on $\RR^n$; then any linear transformation mapping an orthonormal basis with respect to $\langle{-},{-}\rangle$ to the standard basis of $\RR^n$ will conjugate $G$ into $\OO(n)$.

In \cite{LopezShepherd}, Lopez de Gamiz Zearra and the first author characterised when a \GLM{} group is residually finite, subgroup separable and cyclic subgroup separable.  In particular, they showed that $\pi_1(\GG)$ is residually finite if and only if it is either an ascending HNN extension of $\ZZ^n$, or the image of $\Delta_\GG$ is conjugate in $\GL_n(\QQ)$ into $\GL_n(\ZZ)$. Consequently, if $\pi_1(\GG)$ is both $\CAT(0)$ and residually finite, then the image of $\Delta_\GG$ is conjugate into a discrete subgroup of $\OO(n)$, so the image is finite, implying that $\pi_1(\GG)$ is biautomatic.  Therefore, the question of whether every residually finite $\CAT(0)$ group is biautomatic \cite[Question~12.4]{LearyMinasyan} remains open.

In \cite{HughesValiunas,HughesValiunas2}, Hughes and the second author used a variation of the construction of \cite{LearyMinasyan} to give examples of groups that satisfy a variety of non-positive curvature properties (being $\CAT(0)$, hierarchically hyperbolic, injective and asynchronously automatic), yet fail to be biautomatic.

\subsection*{Acknowledgements}

The first author was supported by the CRC 1442 Geometry: Deformations and Rigidity, and by the Mathematics Münster Cluster of Excellence.
The second author was partially supported by the National Science Centre (Poland) grant No.\ 2022/47/D/ST1/00779.

\section{Preliminaries} \label{sec:prelim}

\subsection{Commensurators}\label{subsec:commensurators}

Given a group $H$, its \emph{abstract commensurator} is the group consisting of equivalence classes of isomorphisms between finite-index subgroups of $H$,
\[
\Comm(H) \coloneqq \{ \varphi\colon H' \xrightarrow{\simeq} H'' \mid H',H'' \leq H; [H:H'],[H:H''] < \infty \} \Big/ {\sim},
\]
where the equivalence relation $\sim$ is defined by saying $\varphi_1 \sim \varphi_2$ if $\varphi_1|_{H_0} = \varphi_2|_{H_0}$ for some finite-index subgroup $H_0 \leq H$. If $H \cong \ZZ^n$ for some $n \in \NN$, then $\Comm(H)$ can be identified with $\GL(H \otimes \QQ) \cong \GL_n(\QQ)$ by identifying $\varphi\colon H' \to H''$ with the isomorphism $i'' \circ (\varphi \otimes \QQ) \circ (i')^{-1}$, where isomorphisms $i'\colon H' \otimes \QQ \to H \otimes \QQ$ and $i''\colon H'' \otimes \QQ \to H \otimes \QQ$ are induced by inclusions $H' \to H$ and $H'' \to H$, respectively, and the isomorphism $\varphi \otimes \QQ \colon H' \otimes \QQ \to H'' \otimes \QQ$ is induced by $\varphi$.

Given a group $G$ and a subgroup $H \leq G$, the \emph{commensurator} of $H$ in $G$ is the subgroup
\[
\Comm_G(H) \coloneqq \{ g \in G \mid [H:H \cap g^{\varepsilon}Hg^{-\varepsilon}] < \infty \text{ for } \varepsilon = \pm 1 \} \leq G.
\]
We say that $H$ is \emph{commensurated} in $G$ if $\Comm_H(G) = G$. Note that there exists a canonical map
\begin{align*}
\Phi = \Phi_{G,H}\colon \Comm_G(H) &\to \Comm(H), \\
g &\mapsto [\varphi_g],
\end{align*}
where $\varphi_g\colon H \cap g^{-1}Hg \to H \cap gHg^{-1}$ is defined by $\varphi_g(h) = ghg^{-1}$.

\subsection{Graphs of groups}

A \emph{graph of groups} $\GG$ is a finite connected undirected graph $|\GG|$ together with a collection of \emph{vertex groups} $\{ G_v \mid v \in V\GG \}$, a collection of \emph{edge groups} $\{ G_e \mid e \in E^+\GG \}$ with identifications $G_e = G_{\overline{e}}$ for all $e \in E^+\GG$, and injective group homomorphisms $\{ \iota_e\colon G_e \to G_{i(e)} \mid e \in E^+\GG \}$. Here we write $V\GG$ for the set of vertices of $|\GG|$, $E^+\GG$ for the set of oriented edges in $|\GG|$, and $i(e) \in V\GG$ (respectively $\overline{e} \in E^+\GG$) for the initial vertex (respectively the opposite edge) of $e \in E^+\GG$.

Now, let $\GG$ be a graph of groups, and $T \subseteq |\GG|$ a maximal subtree. The \emph{fundamental group of $\GG$} (with respect to $T$), denoted $\pi_1(\GG,T)$, is a group obtained from the free product $*_{v \in V\GG} G_v$ by adding
\begin{itemize}
    \item generators $\{t_e\mid e\in E^+\GG\}$;
    \item relations of the form $t_e \iota_e(g) t_{\overline{e}} = \iota_{\overline{e}}(g)$ for all $e \in E^+\GG$ and $g \in G_e$ (note that this includes relations $t_e t_{\overline{e}} =1$) and of the form $t_e = 1$ for all $e \in E^+T$.
\end{itemize}
In particular, if the vertex groups $G_v = \langle X_v \mid R_v \rangle$ are finitely presented and the edge groups $G_e = \langle X_e \rangle$ are finitely generated, then $\pi_1(\GG,T)$ has a finite presentation with
\begin{itemize}
    \item generators $E^+\GG \sqcup \bigsqcup_{v \in V\GG} X_v$;
    \item relations of types:
    \begin{itemize}
        \item $t_e w_x t_{\overline{e}} = \overline{w}_x$ for $e \in E^+\GG$ and $x \in X_e$, where $w_x$ (respectively $\overline{w}_x$) is a word over $X_{i(e)}$ (respectively $X_{i(\overline{e})}$) representing $\iota_e(x)$ (respectively $\iota_{\overline{e}}(x)$);
        \item $t_e = 1$ for $e \in E^+T$, and $t_et_{\overline{e}} = 1$ for $e \in E^+\GG$;
        \item $r_v = 1$ for $v \in V\GG$ and $r_v \in R_v$;
    \end{itemize}
\end{itemize}
we call this a \emph{standard presentation} of $\pi_1(\GG,T)$ associated to the presentations $\langle X_v \mid R_v \rangle$ and the generating sets $X_e$. The group $\pi_1(\GG,T)$ does not depend (up to isomorphism) on the choice of $T$ and can be therefore simply denoted by $\pi_1(\GG)$.

The vertex and edge groups of $\GG$ can be canonically identified with subgroups of $\pi_1(\GG)$. The group $\pi_1(\GG)$ acts on its \emph{Bass--Serre tree}, $\TT = \TT_\GG$, such that the vertex (respectively edge) stabilisers under this action are precisely the conjugates of vertex (respectively edge) groups of $\GG$. Moreover, the quotient graph $\TT_\GG / \pi_1(\GG)$ is isomorphic to $|\GG|$.

We say an edge $e \in E^+\GG$ is \emph{ascending} if $\iota_e$ is surjective, and \emph{strictly ascending} if moreover $\iota_{\overline{e}}$ is not surjective. We say $\GG$ is \emph{reduced} if every ascending edge $e \in E^+\GG$ is a loop, i.e.\ satisfies $i(e) = i(\overline{e})$. If $\GG$ is not reduced, then it can be transformed to a reduced graph of groups $\GG'$ by a sequence of \emph{collapse moves}, see \cite{Forester}. Such a transformation does not change $\pi_1(\GG)$ (up to isomorphism) or the set of stabilisers of vertices and ``non-collapsed'' edges under the action on $\TT_\GG$, and the trees $\TT_\GG$ and $\TT_{\GG'}$ are equivariantly quasi-isometric.

\subsection{Commensurating graphs of groups}

We say a graph of groups $\GG$ is \emph{commensurating} if we have $[G_{i(e)} : \iota_e(G_e)] < \infty$ for every $e \in E^+\GG$---or equivalently, if the Bass--Serre tree $\TT_\GG$ is locally finite. In this case, it is easy to see that for any vertex $v \in V\GG$ (or edge $e \in E^+\GG$), the group $H = G_v$ (or $H = G_e$) is commensurated in $G = \pi_1(\GG)$, and the map $\Phi_{G,H}$ defined above is called a \emph{modular homomorphism} for $\GG$ and denoted by $\Delta_{\GG,H}$. Moreover, every vertex and edge group $H'$ of $\GG$ shares a finite-index subgroup with $H$, implying that there exists an isomorphism $\gamma\colon \Comm(H) \to \Comm(H')$ satisfying $\Delta_{\GG,H'} = \gamma \circ \Delta_{\GG,H}$. Therefore, up to such an isomorphism $\gamma$, the modular homomorphism $\Delta_{\GG,H}$ does not depend on the choice of $H$, and we write $\Delta_\GG$ for this map.

It is easy to see that the map $\Delta_\GG$ is unaffected by the collapse moves: that is, if a graph of groups $\GG'$ is obtained from $\GG$ by a sequence of collapse moves, then the canonical isomorphism $\psi\colon \pi_1(\GG') \to \pi_1(\GG)$ satisfies $\Delta_{\GG'} = \Delta_\GG \circ \psi$.

Given $n \in \NN$, we say that a graph of groups $\GG$ is a \emph{\GLM{} graph of groups}, and that $\pi_1(\GG)$ is a \emph{\GLM{} group}, if $G_v \cong \ZZ^n$ for all $v \in V\GG$ and $G_e \cong \ZZ^n$ for all $e \in E^+\GG$. Note that such $\GG$ must be commensurating, since all subgroups of $\ZZ^n$ isomorphic to $\ZZ^n$ have finite index. In this case, since $\Comm(\ZZ^n) \cong \GL_n(\QQ)$, the modular homomorphism is a map $\Delta_\GG\colon \pi_1(\GG) \to \GL_n(\QQ)$, which is uniquely defined up to conjugation in $\GL_n(\QQ)$.

For a field $\FF$ and $n \in \NN$, we write $\AGL_n(\FF)$ for the group of affine linear transformations of $\FF^n$, i.e.\ the semidirect product $\FF^n \rtimes \GL_n(\FF)$, with the subgroups $\FF^n$ and $\GL_n(\FF)$ acting by translations and linear transformations, respectively. We will make use of the following group homomorphism.

\begin{lem} \label{lem:delta}
    Given a \GLM{} graph of groups $\GG$, there exists a group homomorphism
    \[
    \delta_\GG \colon \pi_1(\GG) \to \AGL_n(\QQ),
    \]
    that is injective on the vertex groups of $\GG$, such that the images of the vertex groups of $\GG$ are contained in the group of translations $\QQ^n$, and such that the composite of $\delta_\GG$ with the quotient map $\AGL_n(\QQ) \to \GL_n(\QQ)$ coincides with $\Delta_\GG \colon \pi_1(\GG) \to \GL_n(\QQ)$ (up to conjugacy in $\GL_n(\QQ)$).
\end{lem}

\begin{proof}
    Fix $v \in V\GG$ and an isomorphism $\psi_v \colon G_v \to \ZZ^n$, which is also used to induce an isomorphism $\Comm(G_v) \cong \GL_n(\QQ)$. This allows us to view $\Delta_\GG = \Delta_{\GG,G_v}$ as a homomorphism $\pi_1(\GG) \to \GL_n(\QQ)$, so that for each $h \in \pi_1(\GG)$, the matrix $\Delta_{\GG,G_v}(h) \in \GL_n(\QQ)$ sends $\psi_v(g)$ to $\psi_v(hgh^{-1})$ for any $g \in G_v \cap h^{-1}G_vh$.
    
    Given any $w \in V\GG$, the subgroup $G_v \cap G_w$ has finite index in $G_v$, and therefore for any $g \in G_w$ we have $g^m \in G_v$ for some $m = m(g) \geq 1$. We then define $\delta_\GG$ on $G_w$ by setting $\delta_\GG(g) = (\frac{1}{m(g)}\psi_v(g^{m(g)}),\id)$; note that, since $\QQ^n$ is uniquely divisible, this definition is independent of the value of $m(g)$ chosen. Given $e \in E^+\GG$, we also set $\delta_\GG(t_e) = (0,\Delta_{\GG,G_v}(t_e))$.  It is then clear, by construction, that $\delta_\GG$ is injective on the vertex groups and sends them into the translation subgroup $\QQ^n$, and that the image of $\delta_\GG(\pi_1(\GG))$ under the quotient map $\AGL_n(\QQ) \to \GL_n(\QQ)$ is $\Delta_\GG(\pi_1(\GG))$.  It remains to check that $\delta_\GG$ extends to a group homomorphism.

    The relations of the form $t_e = 1$ (for $e \in E^+T$) are clearly preserved by $\delta_\GG$, since $\Delta_\GG(t_e) = \Delta_\GG(1) = \id$. On the other hand, given any $e \in E^+\GG$ and $g \in G_e$, there exists a constant $m' \geq 1$ such that $\iota_e(g)^{m'},\iota_{\overline{e}}(g)^{m'} \in G_v$, and therefore
    \begin{align*}
    \delta_\GG(t_e) \cdot \delta_\GG(\iota_e(g)) \cdot \delta_\GG(t_{\overline{e}})
    &= \textstyle (0,\Delta_\GG(t_e)) \cdot (\frac{1}{m'}\psi_v(\iota_e(g^{m'})),\id) \cdot (0,\Delta_\GG(t_{\overline{e}})) \\
    &= \textstyle (\frac{1}{m'}\Delta_\GG(t_e) \cdot \psi_v(\iota_e(g^{m'})),\Delta_\GG(t_e)\Delta_\GG(t_{\overline{e}})) \\
    &= \textstyle (\frac{1}{m'} \psi_v(t_e\iota_e(g^{m'})t_e^{-1}),\Delta_\GG(t_et_{\overline{e}})) \\
    &= \textstyle (\frac{1}{m'} \psi_v(\iota_{\overline{e}}(g^{m'})),\id)\\
    &= \delta_\GG(\iota_{\overline{e}}(g)),
    \end{align*}
    implying that $\delta_\GG$ preserves the relation $t_e\iota_e(g)t_{\overline{e}} = \iota_{\overline{e}}(g)$ as well.  Thus $\delta_\GG$ extends to a homomorphism, as needed.    
\end{proof}

\subsection{CAT(0) spaces and groups}

Recall that given a geodesic triangle $\Delta$ in a geodesic metric space $X$, a \emph{comparison triangle} $\overline\Delta$ is a geodesic triangle on the Euclidean plane $\EE^2$ having the same edge lengths as $\Delta$.  In this situation, there exists a continuous map $f_\Delta\colon \Delta \to \overline\Delta$ that is an isometry on each of the edges.  We say that the space $X$ is \emph{CAT(0)} if for any geodesic triangle $\Delta$ in $X$ and any $x,y \in \Delta$, we have $d_X(x,y) \leq d_{\EE^2}(f_\Delta(x),f_\Delta(y))$. In this note, we are only using the fact that the Cartesian product $\EE^n \times T$ with the product metric is $\CAT(0)$ when $T$ is a simplicial tree.

We say that a finitely generated group $G$ is \emph{CAT(0)} if it acts properly and cocompactly on a $\CAT(0)$ space $X$ by semisimple isometries.

\subsection{Biautomatic groups}

Here we briefly recall the definition of biautomaticity.  We will not use this definition directly, instead relying on known results.

Let $G$ be a group with a finite generating set $X \subseteq G$ such that $1 \in X = X^{-1}$, and write $X^*$ for the free monoid on $X$. Given $\LL \subseteq X^*$, we say that $(X,\LL)$ is a \emph{biautomatic structure} on $G$ if
\begin{itemize}
    \item $\LL$ is a \emph{regular language}, i.e.\ recognised by a finite state automaton over $X$;
    \item the function $\pi|_\LL$ is finite-to-one and surjective, where $\pi\colon X^* \to G$ is the canonical surjection;
    \item $(X,\LL)$ satisfies the \emph{two-sided fellow traveller property}: there exists a constant $\kappa \geq 0$ such that for any $x,y \in X$ and $u,v \in \LL$ satisfying $xu = vy$ in $G$, we have $d_{\Cay(G,X)}(xu_t,v_t) \leq \kappa$ for all $t \in \NN$, where given $w \in X^*$ we write $w_t$ for the initial subword of $w$ of length $t$ (or the whole of $w$ if $t > |w|$).
\end{itemize}
We say that $G$ is \emph{biautomatic} if it has a biautomatic structure.

Note that the definition given here differs slightly from the definition appearing usually in the literature, as we require the map $\pi|_\LL$ to be finite-to-one.  However, as pointed out in \cite{Amrhein}, such a restriction is necessary in order to satisfy the original definition of biautomatic groups in terms of multiplier automata \cite[\S2]{Epstein}, whereas we lose no generality by assuming $\pi|_\LL$ to be finite-to-one since any biautomatic group $G$ has a biautomatic structure $(X,\LL')$ with $\pi|_{\LL'}$ bijective \cite[Theorem~2.5.1]{Epstein}.

\section{Biautomaticity}

In the case $\Delta_\GG$ has finite image, the following will be used to show that $\pi_1(\GG)$ is biautomatic.

\begin{lem} \label{lem:normal-action-trivial}
    Let $\GG$ be a \GLM{} graph of groups, and suppose that $\Delta_\GG$ has finite image. Then $\pi_1(\GG)$ has a normal subgroup isomorphic to $\ZZ^n$ acting trivially on $\TT_\GG$.
\end{lem}

\begin{proof}
    By \cite[Proposition~2.4]{LopezShepherd}, it is enough to show that $G \coloneqq \Delta_\GG(\pi_1(\GG)) < \GL_n(\QQ)$ is conjugate to a subgroup of $\GL_n(\ZZ)$.  But indeed, consider the lattice $H_0 = \ZZ^n < \QQ^n$ (where by \emph{lattice} we mean a subgroup isomorphic to $\ZZ^n$, i.e.\ a discrete subgroup of $\QQ^n$ that spans $\QQ^n$ over $\QQ$). Then $H \coloneqq \bigcap_{A \in G} A(H_0)$ is a lattice in $\QQ^n$ (since any intersection of two lattices is a lattice and since $G$ is finite), and by construction, $H$ is $G$-invariant. Any linear transformation mapping a basis of $H$ (as a free abelian group) to the standard basis of $\QQ^n$ will then conjugate $G$ to a subgroup of $\GL_n(\ZZ)$, as required.
\end{proof}

\begin{proof}[Proof of Theorem~\ref{thm:main}\ref{it:main-biaut}]
    The implication \ref{iit:main-biaut} $\Rightarrow$ \ref{iit:main-sbiaut} is trivial.  It is thus enough to show the implications \ref{iit:main-sbiaut} $\Rightarrow$ \ref{iit:main-finite} $\Rightarrow$ \ref{iit:main-biaut} and \ref{iit:main-biaut} $\Leftrightarrow$ \ref{iit:main-vbiaut}.

    Suppose first that $G = \pi_1(\GG)$ is a subgroup of a biautomatic group $G'$, and let $H \leq G$ be a vertex group of $\GG$.  Then $H$ is a finitely generated abelian group that is commensurated in $G$, and therefore $G \leq \Comm_{G'}(H)$. By \cite[Theorem~1.2]{Valiunas}, the image of $\Phi\colon \Comm_{G'}(H) \to \Comm(H) \cong \GL_n(\QQ)$ is finite. Since the modular homomorphism $\Delta_\GG$ is precisely the restriction of $\Phi$ to $G$, it follows that the image of $\Delta_\GG$ is finite as well.  This shows \ref{iit:main-sbiaut} $\Rightarrow$ \ref{iit:main-finite}.

    Suppose now that $\Delta_\GG$ has finite image, and let $N \unlhd \pi_1(\GG)$ be the normal subgroup acting trivially on $\TT_\GG$, given by Lemma~\ref{lem:normal-action-trivial}. We then have an induced action of $F = \pi_1(\GG)/N$ on $\TT_\GG$. Since $N$ has finite index in every vertex stabiliser under the $\pi_1(\GG)$-action, the action of $F$ has finite vertex stabilisers; we also have that the $F$-action is cocompact since the $\pi_1(\GG)$-action is.  It follows that $F$ is virtually free \cite[Corollary~IV.1.9]{DicksDunwoody}, in particular word-hyperbolic and thus biautomatic \cite[Theorem~3.4.5]{Epstein}.

    Now consider the homomorphism $\delta_\GG\colon \pi_1(\GG) \to \AGL_n(\QQ)$ given by Lemma~\ref{lem:delta}.  By construction, $\delta_\GG$ is injective on $N$, whereas the image of $\delta_\GG$ is virtually abelian (since the fact that $\Delta_\GG$ has finite image means a finite-index subgroup of the image of $\delta_\GG$ is contained in the translation subgroup $\QQ^n$ of $\AGL_n(\QQ)$). Consider the homomorphism $\Psi\colon \pi_1(\GG) \to H \times F$ given by $\Psi(g) = (\delta_\GG(g),gN)$, where $H$ is the image of $\delta_\GG$.  Since $\ker(\delta_\GG) \cap N = \{1\}$ it follows that $\Psi$ is injective, and it is also clear that the image of $\Psi$ surjects onto each factor.  It follows that $\pi_1(\GG)$ is isomorphic to a subdirect product of $H \times F$, where $H$ is virtually abelian and $F$ is biautomatic. Thus $\pi_1(\GG)$ is biautomatic by \cite[Proposition~8.2]{LearyMinasyan}; this proves \ref{iit:main-finite} $\Rightarrow$ \ref{iit:main-biaut}.

    Finally, note that if $G' \leq \pi_1(\GG)$ is a finite-index subgroup, then the action of $G'$ on $\TT_\GG$ gives a description of $G'$ as a fundamental group of a \GLM{} graph of groups $\GG'$.  This description is consistent with the modular homomorphisms, i.e.\ the restriction $\Delta_\GG|_{G'}$ coincides with $\Delta_{\GG'}$ (up to conjugacy in $\GL_n(\QQ)$); it follows that the image of $\Delta_{\GG'}$ is (conjugate to) a finite-index subgroup of the image of $\Delta_\GG$.  In particular, $\Delta_{\GG'}$ has finite image if and only if $\Delta_\GG$ does; by the equivalence of \ref{iit:main-biaut} and \ref{iit:main-finite} shown above, it then follows that \ref{iit:main-biaut} $\Leftrightarrow$ \ref{iit:main-vbiaut}, as required.
\end{proof}

\begin{rem} \label{rem:using-LM-paper}
    One can show the equivalence of parts \ref{iit:main-biaut}, \ref{iit:main-vbiaut} and \ref{iit:main-finite} in Theorem~\ref{thm:main}\ref{it:main-biaut} directly from \cite[Theorem~1.2]{LearyMinasyan}, without using the full power of its generalisation \cite[Theorem~1.2]{Valiunas}.  Indeed, in order to show the implication \ref{iit:main-biaut} $\Rightarrow$ \ref{iit:main-finite}, one may notice that for any reduced \GLM{} graph of groups $\GG$:
    \begin{itemize}
        \item either $\GG$ has a strictly ascending loop, in which case one can show that the Dehn function for $\pi_1(\GG)$ is exponential---and thus $\pi_1(\GG)$ is not biautomatic, not even combable \cite[Theorem~3.6.6]{Epstein};
        \item or $\GG$ has a non-ascending edge $e$ such that $\overline{e}$ is also non-ascending, in which case one can show that $G_e$ is the centraliser of a finitely generated subgroup of $\pi_1(\GG)$;
        \item or $\GG$ has a single vertex and all its edge maps are isomorphisms, in which case $\pi_1(\GG) = H \rtimes F$ where $H \cong \ZZ^n$ and $F$ is a finitely generated free group---then, by analysing the kernel of the map $F \to \Aut(H)$, one can show that if $\Delta_\GG$ has infinite image, then $H$ is either self-centralising or the centraliser of a subgroup $\langle f_1,f_2 \rangle \leq F$.
    \end{itemize}
    In particular, in the latter two cases there exists a vertex or edge group $H$ of $\GG$ that coincides with the centraliser of a finite subset of $\pi_1(\GG)$, and so is $\LL$-quasiconvex with respect to any biautomatic structure $(X,\LL)$ on $\pi_1(\GG)$. It then follows from \cite[Theorem~1.2]{LearyMinasyan} that $\Delta_\GG$ has finite image, as required.
\end{rem}

\section{Non-positive curvature}

We now turn to the proof of Theorem \ref{thm:main}\ref{it:main-CAT0}. The argument is essentially the same as for the proof of Theorem 7.2 in \cite{LearyMinasyan}, but with different notation that allows us to work with a general graph of groups $\GG$ rather than just an HNN extension.

\begin{thm}\cite[Theorem 6.4($4'$)]{LearyMinasyan}\label{thm:innerproduct}
    Let $L$ be a finite-rank free abelian group acting properly by semisimple isometries on a CAT(0) space $X$.
    Then there is an inner product $\langle\cdot,\cdot\rangle_L$ on $L\otimes\QQ$, such that, for any isometry $\varphi$ of $X$ that commensurates $L$, the inner product $\langle\cdot,\cdot\rangle_L$ is preserved by the image of $\varphi$ in $\GL(L \otimes \QQ)$ (see Subsection \ref{subsec:commensurators}).
\end{thm}

\begin{proof}[Proof of Theorem~\ref{thm:main}\ref{it:main-CAT0}]
First suppose that the image of $\Delta_\GG$ is conjugate in $\GL_n(\RR)$ into $\OO(n)$.
In particular, there exists $P\in\GL_n(\RR)$ with $P\Delta_\GG(\pi_1(\GG))P^{-1}\subset\OO(n)$. Consider the map $\delta_\GG\colon \pi_1(\GG) \to \AGL_n(\QQ)$ given by Lemma~\ref{lem:delta}, and let $\alpha\colon \AGL_n(\QQ) \to \AGL_n(\RR)$ and $\gamma_P\colon \AGL_n(\RR) \to \AGL_n(\RR)$ be the inclusion and conjugation by $(0,P)$, respectively. Then the map $\beta = \gamma_P \circ \alpha \circ \delta_\GG \colon \pi_1(\GG) \to \AGL_n(\RR)$ has its image contained in $\RR^n \rtimes \OO(n)$, i.e.\ the group of isometries of the Euclidean $n$-space $\EE^n$. Note that, since the restriction of $\delta_\GG$ to any vertex group $G_v$ is an injective map to the group $\QQ^n < \AGL_n(\QQ)$, the restriction of $\beta$ to $G_v$ gives a proper and cocompact action of $G_v$ on $\EE^n$ by translations.

Now define an isometric action of $\pi_1(\GG)$ on $\EE^n\times \TT_\GG$ (which is a $\CAT(0)$ space with the product metric), where $\pi_1(\GG)$ acts on the $\EE^n$ factor according to $\beta$ and on the $\TT_\GG$ factor according to the usual Bass--Serre tree action.
Observe that the action of $\pi_1(\GG)$ on $\EE^n\times \TT_\GG$ is proper and cocompact since each vertex stabiliser for $\TT_\GG$ is a conjugate of a vertex group $G_v$, and each vertex group acts properly and cocompactly on $\EE^n$.
Thus $\pi_1(\GG)$ is a CAT(0) group.

Now suppose that $\pi_1(\GG)$ is CAT(0).
Consider a vertex group $G_v$, and consider the modular homomorphism as a map $\Delta_\GG:\pi_1(\GG)\to \Comm(G_v) \cong \GL(G_v\otimes\QQ)\cong\GL_n(\QQ)$.
We know that $\pi_1(\GG)$ commensurates $G_v$, so Theorem \ref{thm:innerproduct} implies that the image of $\Delta_\GG$ preserves the inner product $\langle\cdot,\cdot\rangle_L$ on $G_v\otimes\QQ$.
Equivalently, considering the image of $\Delta_\GG$ as a subgroup of $\GL_n(\QQ)$, this implies that the image of $\Delta_\GG$ is conjugate in $GL_n(\RR)$ into $\OO(n)$.
\end{proof}

\bibliographystyle{amsalpha}
\bibliography{ref3}

\end{document}